\documentclass{amsart}
\usepackage[utf8]{inputenc}
\usepackage[english]{babel}
\usepackage{lmodern}
\usepackage{enumitem}
\usepackage{mathtools}
\usepackage{amsfonts}
\usepackage{amsthm}
\usepackage{listings}
\usepackage{graphicx}
\usepackage{tikz}
\usepackage{todonotes}
\usepackage[toc,page]{appendix}
\usepackage{xcolor}
\usepackage[colorlinks,
    linkcolor={red!50!black},
    citecolor={blue!50!black},
    urlcolor={blue!80!black}]{hyperref}
\usepackage{import}

\newtheorem{theorem}{Theorem}[section]
\newtheorem{corollary}[theorem]{Corollary}
\newtheorem{definition}[theorem]{Definition}
\newtheorem{lemma}[theorem]{Lemma}
\newtheorem{example}[theorem]{Example}

\newtheorem{remark}[theorem]{Remark}

\newcommand{\Z}{\mathbb{Z}}
\newcommand{\Q}{\mathbb{Q}}

\newcommand{\C}{\mathbb{C}}
\newcommand{\lk}{\operatorname{lk}}
\newcommand{\fr}{\operatorname{fr}}
\newcommand{\writhe}{\operatorname{wr}}
\newcommand{\sgn}{\operatorname{sgn}}

\title{A new polynomial criterion for periodic knots}
\author{Maciej Markiewicz}
\address{Institute of Mathematics, University of Warsaw, ul. Banacha 2, 02-097 Warsaw, Poland}
\email{ma.markiewicz3@uw.edu.pl}
\thanks{MM ORCID: 0000-0002-8743-2476}

\author{Wojciech Politarczyk}
\address{Institute of Mathematics, University of Warsaw, ul. Banacha 2, 02-097 Warsaw, Poland}
\email{wpolitarczyk@mimuw.edu.pl}
\thanks{WP ORCID: 0000-0002-4844-4327}

\makeatletter
\@namedef{subjclassname@2010}{\textup{2010} Mathematics Subject Classification}
\makeatother

\begin{document}

\maketitle

\begin{abstract}
  The purpose of this paper is to present a new periodicity criterion.
  For that purpose, we study the HOMFLY-PT polynomial and the Kauffman polynomial of cables of periodic links.
  Furthermore, we exhibit a couple of examples for which our criterion is stronger than many previously know criteria, among which is the Khovanov homology criterion of Borodzik and the second author.
\end{abstract}
\section{Introduction}

We say that a link is $r$-periodic, for some \(r>1\), if it is invariant under a semi-free action of the finite cyclic group~$\mathbb{Z}_{r}$ on~$S^{3}$ and disjoint from the fixed-point set.
It is natural to ask whether a given link invariant can be used to detect the existence of a periodic symmetry of a link or lack thereof.
In particular, polynomial link invariants were the first choice for obtaining periodicity criteria.
For example, \cite{davisAlexanderPolynomialsEquivariant2006,davisAlexanderPolynomialsPeriodic1991,Mu} obstruct periodicity with the aid of the Alexander polynomial.
Similarly, many periodicity criteria use the Jones polynomial~\cite{yokotaJonesPolynomialPeriodic1991a,murasugiJonesPolynomialsPeriodic1988} and HOMFLY-PT polynomial~\cite{Tr,Prz}.
For a comprehensive survey  polynomial invariants of periodic links, refer to~\cite{przytycki-survey}.
Similarly, one can use link homology theories to devise periodicity criteria~\cite{politarczykEquivariantJonesPolynomials2017,BP,JN}.

The main goal of the paper is to give new periodicity criteria in terms of the HOMFLY-PT~polynomial~\cite{HOMFLY,PT} and the Kauffman~polynomial~\cite{Ka}.
Our approach is based on an observation that if a knot \(K\) is periodic, then so are suitable cables of \(K\).

Before we state the main theorem, we recall the skein relations defining the HOMFLY-PT polynomial and the Kauffman polynomial to avoid any confusion resulting from different conventions existing in the literature.
Let \(\mathcal{R}\) denote the subring of the Laurent polynomial ring \(\Z[a^{\pm1},z^{\pm1}]\) generated by \(a^{\pm1},z, \frac{a+a^{-1}}{z}\).
The \emph{HOMFLY-PT polynomial}~\(P_{L}(a,z) \in \mathcal{R}\), is the invariant of oriented links characterized by the following properties \cite[Theorem 15.2]{Lickorish}
\begin{align*}
  P_{U}(a,z) &= 1, \\
  aP_{L_{+}}(a,z)+a^{-1}P_{L_{-}}(a,z) &= zP_{L_{0}}(a,z). 
\end{align*}
Above, \(U\) denotes the unknot and $L_{+},L_{-},L_{0}$ are links which are identical except in a neighborhood of a single crossing where they look as depicted in~Figure~\ref{fig:skein-relation}.
For a diagram $D$ of a link $L$, the \emph{Kauffman polynomial} $\Lambda_{D}\in\mathcal{R}$ satisfies \cite[Theorem 15.3, 15.5]{Lickorish}
\begin{align}
  \Lambda_{U}(a,z) &=1, \nonumber \\
  \Lambda_{D_{+}}(a,z) + \Lambda_{D_{-}}(a,z)&=z(\Lambda_{D_{0}}(a,z)+\Lambda_{D_{\infty}}(a,z)), \nonumber\\
  \Lambda_{D_{\pm1}}&=a^{\pm1}\Lambda_{D}. \nonumber
\end{align}
Above, we denote by $U$ the unknot and $D_{+},D_{-},D_{0},D_{\infty},D_{\pm1}$ are diagrams which locally look as in Figure~\ref{fig:kauffman-relation} and are identical otherwise.
The \emph{Kauffman polynomial} of a link $L$ is then defined by 
\[F_{L}(a,z)=a^{-\writhe_{D}(L)}\Lambda_{D}(a,z),\]
for any diagram $D$ of $K$, where $\writhe_{D}(L)$ is the writhe of the diagram.

\begin{figure}[t]
  \centering
  \def\svgscale{0.3}
  \begingroup%
  \makeatletter%
  \providecommand\color[2][]{%
    \errmessage{(Inkscape) Color is used for the text in Inkscape, but the package 'color.sty' is not loaded}%
    \renewcommand\color[2][]{}%
  }%
  \providecommand\transparent[1]{%
    \errmessage{(Inkscape) Transparency is used (non-zero) for the text in Inkscape, but the package 'transparent.sty' is not loaded}%
    \renewcommand\transparent[1]{}%
  }%
  \providecommand\rotatebox[2]{#2}%
  \newcommand*\fsize{\dimexpr\f@size pt\relax}%
  \newcommand*\lineheight[1]{\fontsize{\fsize}{#1\fsize}\selectfont}%
  \ifx\svgwidth\undefined%
    \setlength{\unitlength}{827.81680462bp}%
    \ifx\svgscale\undefined%
      \relax%
    \else%
      \setlength{\unitlength}{\unitlength * \real{\svgscale}}%
    \fi%
  \else%
    \setlength{\unitlength}{\svgwidth}%
  \fi%
  \global\let\svgwidth\undefined%
  \global\let\svgscale\undefined%
  \makeatother%
  \begin{picture}(1,0.33846169)%
    \lineheight{1}%
    \setlength\tabcolsep{0pt}%
    \put(0,0){\includegraphics[width=\unitlength,page=1]{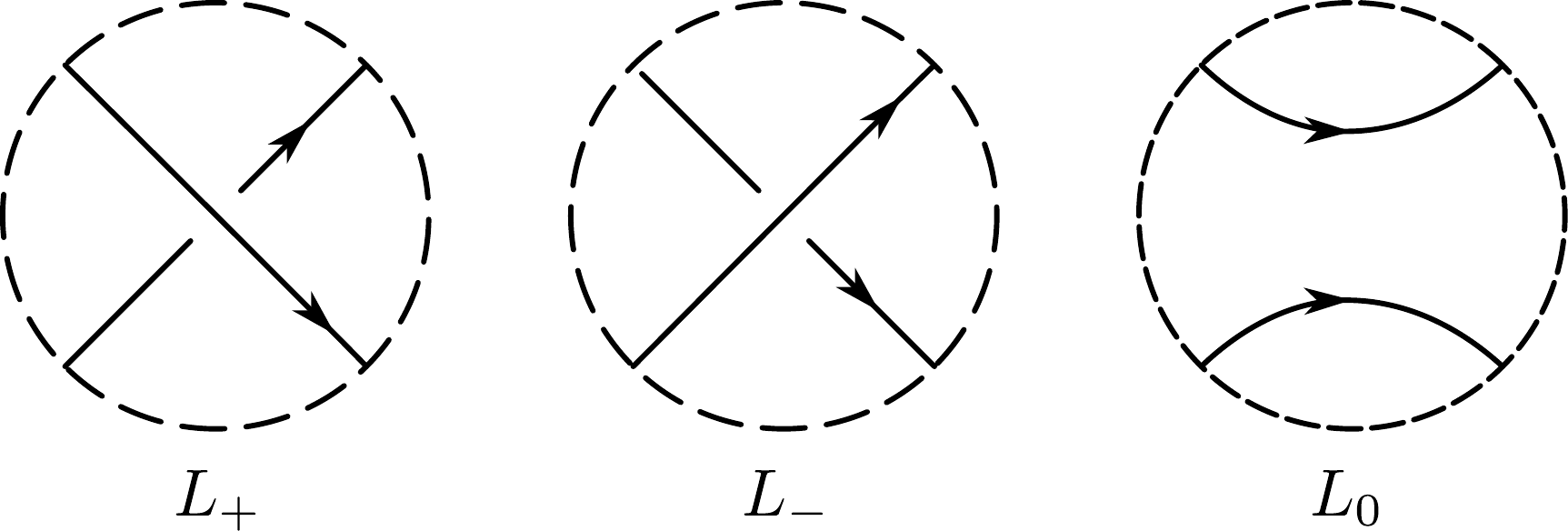}}%
  \end{picture}%
  \endgroup%
  \caption{Positive crossing, negative crossing, and the oriented resolution, respectively.}
  \label{fig:skein-relation}
\end{figure}

\begin{figure}
  \centering
  \def\svgscale{0.3}
  \begingroup%
  \makeatletter%
  \providecommand\color[2][]{%
    \errmessage{(Inkscape) Color is used for the text in Inkscape, but the package 'color.sty' is not loaded}%
    \renewcommand\color[2][]{}%
  }%
  \providecommand\transparent[1]{%
    \errmessage{(Inkscape) Transparency is used (non-zero) for the text in Inkscape, but the package 'transparent.sty' is not loaded}%
    \renewcommand\transparent[1]{}%
  }%
  \providecommand\rotatebox[2]{#2}%
  \newcommand*\fsize{\dimexpr\f@size pt\relax}%
  \newcommand*\lineheight[1]{\fontsize{\fsize}{#1\fsize}\selectfont}%
  \ifx\svgwidth\undefined%
    \setlength{\unitlength}{1127.05919803bp}%
    \ifx\svgscale\undefined%
      \relax%
    \else%
      \setlength{\unitlength}{\unitlength * \real{\svgscale}}%
    \fi%
  \else%
    \setlength{\unitlength}{\svgwidth}%
  \fi%
  \global\let\svgwidth\undefined%
  \global\let\svgscale\undefined%
  \makeatother%
  \begin{picture}(1,0.51255441)%
    \lineheight{1}%
    \setlength\tabcolsep{0pt}%
    \put(0,0){\includegraphics[width=\unitlength,page=1]{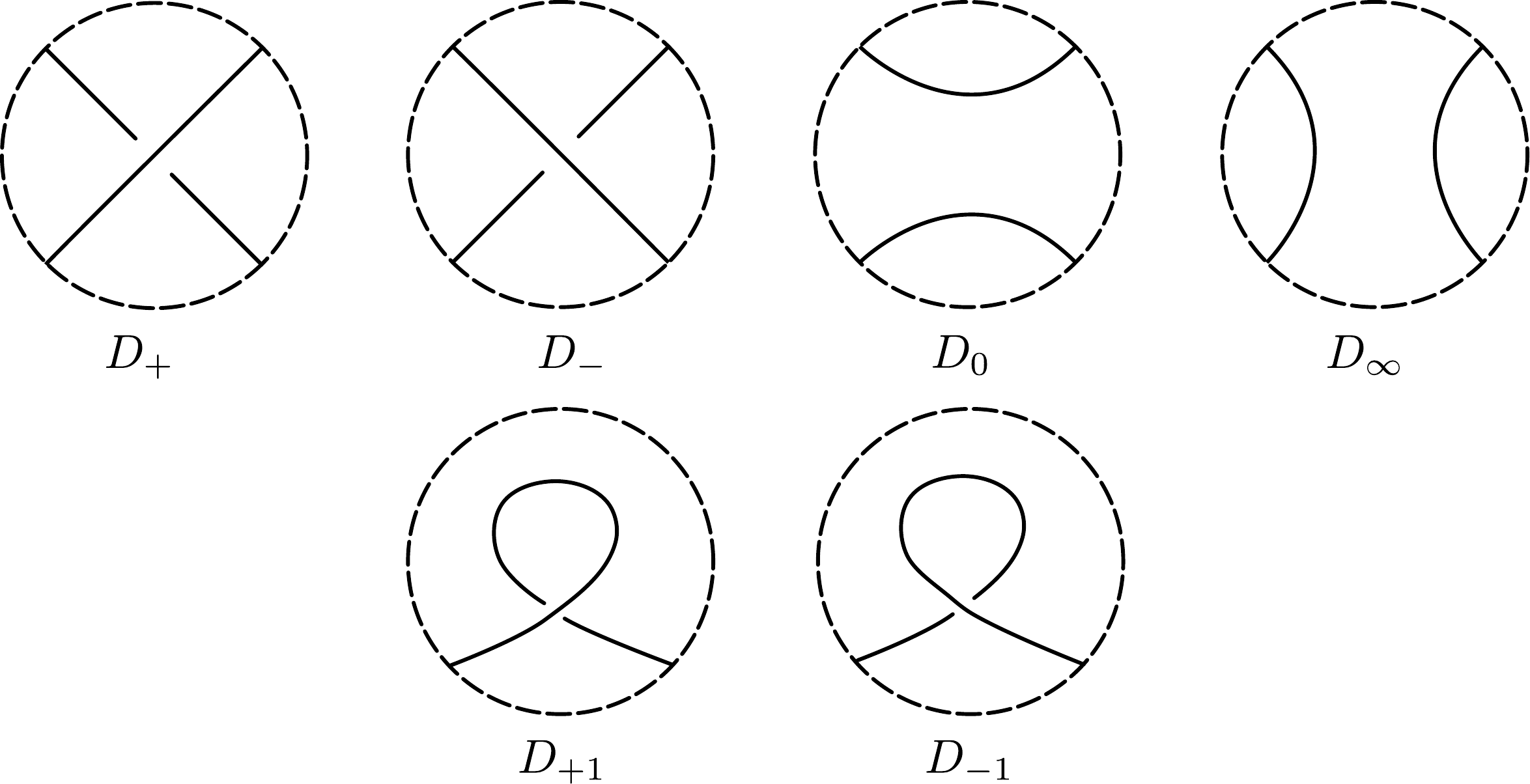}}%
  \end{picture}%
  \endgroup%
  \caption{Link diagrams involved in the definition of Kauffman polynomial.}
  \label{fig:kauffman-relation}
\end{figure}

Let \(K\) be an \(r\)-periodic knot, for some \(r>1\), and suppose that \(K\) is equipped with a framing, i.e., a diffeomorphism \(\varphi \colon S^{1} \times D^{2} \to N(K) \subset S^{3}\) and \(\varphi(S^1\times \{0\})=K\), where \(N(K)\) is a \(\Z_{r}\)-invariant closed tubular neighborhood of \(K\).
We say that \(\varphi\) is \(\Z_{r}\)-\emph{equivariant} if the action induced by \(\varphi\) on the solid torus is trivial on the \(D^{2}\) factor.
For any \(m \geq 1\), we will denote by \(K^{(m)}_{\varphi}\) the \(\varphi\)-\emph{framed} \(m\)-\emph{cable} of \(K\), i.e., the \(m\)-component cable link of \(K\) formed by the image of \(m\) parallel and unlinked copies of the core of the solid torus via \(\varphi\), see Section~\ref{sec:periodic-cables} for the precise definition.

For a prime \(p\) and an integer \(n>0\) let \(\mathcal{J}_{p^{n}}\) be the ideal in \(\mathcal{R}\) generated by the monomials
\[p^{i} z^{p^{n-i}}, \quad \text{for } i=0,1,\ldots,n.\]

\begin{theorem}\label{thm:main-thm}
  Suppose that \(p\) is a prime and \(n>0\).
  Let $K$ be a $p^{n}$-periodic knot equipped with an equivariant framing \(\varphi\).
  Then the following congruences are satisfied
  \begin{align*}
    P_{K^{(m)}_{\varphi}}(a,z) &\equiv \left(\frac{a+a^{-1}}{z}\right)^{m-1} \cdot P_{K}(a,z)^{m} \pmod{\mathcal{J}_{p^{n}}}, \\
    F_{K^{(m)}_{\varphi}}(a,z) &\equiv d^{m-1} \cdot F_{K}(a,z)^{m} \pmod{\mathcal{J}_{p^{n}}},
  \end{align*}
  where $d=z^{-1}(a+a^{-1})-1$.
\end{theorem}

\begin{remark}
  Observe that the congruences in Theorem~\ref{thm:main-thm} hold in the quotient ring~\(\mathcal{R} / \mathcal{J}_{p^n}\) and not in~\(\Z[a^{\pm1},z^{\pm1}] / \mathcal{J}_{p^n}\Z[a^{\pm1},z^{\pm1}]\).
  Indeed, observe that \(z^{p^n} \in \mathcal{J}_{p^n}\) is invertible in \(\Z[a^{\pm1},z^{\pm1}]\), hence the latter quotient is trivial.
\end{remark}

As an immediate corollary, we obtain:

\begin{corollary}\label{cor:independence-of-framings}
  Suppose that \(p\) is a prime and \(n>0\).
  Let \(K\) be a \(p^{n}\)-periodic knot and let \(\varphi_{1}\) and \(\varphi_{2}\) be two equivariant framings of \(K\), then for any \(m>1\),
  \begin{align*}
    P_{K^{(m)}_{\varphi_{1}}}(a,z) &\equiv P_{K^{(m)}_{\varphi_{2}}}(a,z) \pmod{\mathcal{J}_{p^{n}}}, \\
    F_{K^{(m)}_{\varphi_{1}}}(a,z) &\equiv F_{K^{(m)}_{\varphi_{2}}}(a,z) \pmod{\mathcal{J}_{p^{n}}}.
  \end{align*}
\end{corollary}

Recall that by setting \(a = -iq^{-2}\) and \(z = i(q^{-1}-q)\) in the HOMFLY-PT polynomial, we recover the~\emph{reduced Jones polynomial}, \(J_{L}(q) \in \Z[q^{\pm1}]\), which is characterized by the following properties
\begin{align}
  J_{U}(q) &= 1, \\
  q^{-2}J_{L_{+}}(q) - q^{2}J_{L_{-}}(q) &= (q^{-1}-q)J_{L_{0}}(q), \label{eq:Jo-skein-relation} 
\end{align}
Therefore, Theorem~\ref{thm:main-thm} specializes to the following corollary.

\begin{corollary}\label{cor:JoThm}
  Let $K$ be a $p^{n}$-periodic knot with an equivariant framing \(\varphi\).
  Suppose that $\mathcal{I}_{p^{n}}$ is the ideal in $\Z[q^{\pm1}]$ generated by monomials
  \[p^{n}(q^{2}-1),p^{n-1}(q^{2}-1)^{p},...,p(q^{2}-1)^{p^{n-1}},(q^{2}-1)^{p^{n}}.\]
  Then, for any \(m>1\), the following congruence is satisfied
  \[J_{K^{(m)}_{\varphi}}(q)\equiv (q+q^{-1})^{m-1}J_{K}(q)^{m} \pmod{\mathcal{I}_{p^{n}}}.\]
\end{corollary}

Note that Theorem~\ref{thm:main-thm}, Corollary~\ref{cor:independence-of-framings} and Corollary~\ref{cor:JoThm} give an easily-applicable periodicity criterion provided that we can identify candidates for equivariant framings.
This characterization is given in Lemma~\ref{lem:admissible-framings}.
Consequently, Theorem~\ref{thm:main-thm}, Corollary~\ref{cor:independence-of-framings} or Corollary~\ref{cor:JoThm} coupled with Lemma~\ref{lem:admissible-framings} can be applied to rule out the possible periods in some cases.
In particular, these criteria can be applied to prove that the knots~\(12n749\) and~\(15n124640\) are not \(7\)-periodic.
These examples are interesting because many previously known criteria, including the criterion using the Khovanov homology~\cite{BP}, were ineffective in these cases.
For details refer to Section~\ref{sec:comp-with-other}.

Many periodicity criteria are ineffective for small periods, see~\cite[Lemma 4.5]{Prz}.
Unfortunately, Corollary~\ref{cor:JoThm} has also some limitations.
To be more precise, the criterion is ineffective for periods \(2\) and \(3\).
These limitations are discussed in detail in Section~\ref{sec:limitations}.

The paper is organized as follows.
In Section~\ref{sec:periodic-cables}, we discuss properties of cables of periodic links and prove Theorem~\ref{thm:main-thm}.
In Section~\ref{sec:comp-with-other} we verify that the knots~\(12n749\) and~\(15n124640\) are not \(7\)-periodic.
In Section~\ref{sec:limitations} we discuss limitations of our periodicity criteria.
Appendix~\ref{sec:code} contains the Sage source code that we used for computations in Section~\ref{sec:comp-with-other}.

\subsection*{Acknowledgments}
The paper is an abridged version of the Master's Thesis of the first author prepared under the supervision of the second author.
We are also grateful to Maciej Borodzik for his suggestion for improving the manuscript.
MM and WP were supported by the National Science Center grant 2016/22/E/ST1/00040.

\section{Periodic cables of periodic knots and the proof of main theorem}
\label{sec:periodic-cables}

Throughout the paper we assume, unless otherwise stated, that all knots are oriented.

\begin{definition}
  Let \(r>1\) be a positive integer.
  We say that a knot $K$ in $S^{3}$ is \emph{$r$-periodic} if it is invariant under a semi-free action of \(\Z_{r}\), i.e., the finite cyclic group of order \(r\), on $S^{3}$.
  Moreover, we require that $K$ is disjoint from the fixed point set.
\end{definition}

Because of the resolution of the Smith conjecture~\cite{Mo}, a knot is \(r\)-periodic if it admits an \emph{\(r\)-periodic diagram}, i.e., a diagram that is invariant under a rotation of~$\mathbb{R}^{2}$ of order~$r$.

Let \(K\) be a knot and let \(\varphi\) be a framing of \(K\), i.e., a diffeomorphism~\(\varphi \colon S^{1} \times D^{2} \to N(K)\) identifying a closed tubular neighbourhood of~\(K\) with a solid torus.
Notice that~\(\varphi\) determines a longitude \(\lambda = \varphi(S^{1} \times \{1\})\) for~\(K\).
Define the \emph{framing coefficient} \(\fr(\varphi) = \lk(K,\lambda)\).
It is a standard fact that the framing coefficient characterizes~\(\varphi\) up to isotopy, i.e.,\ \(\varphi_{1}\) is isotopic to \(\varphi_{2}\) if, and only if, \(\fr(\varphi_{1}) = \fr(\varphi_{2})\).

\begin{example}
  Let~\(D\) be a diagram of \(K\).
  Recall that~\(D\) determines a framing of~\(K\), the~\emph{blackboard framing} of~\(D\), denoted by~\(\varphi_{D}\).
  For the blackboard framing we have \(\fr(\varphi_{D}) = \writhe_{D}(K)\), where~\(\writhe_{D}(K)\) denotes the \emph{writhe} of~\(D\), i.e.,\ the signed count of crossings of~\(D\), see Figure~\ref{fig:skein-relation}.
\end{example}

For~\(m>1\), the \emph{\(\varphi\)-framed \(m\)-cable of \(K\)} is the \(m\)-component link
\[K^{(m)}_{\varphi} = \varphi\left(S^{1} \times \{0,1/m,2/m,\ldots,(m-1)/m\} \right) \subset N(K) \subset S^3.\]
We orient components of \(K^{(m)}_{\varphi}\) in such a way that the orientation of \(\varphi(S^{1}\times\frac{k}{m})\), for \(k=0,1,\ldots,m-1\), agrees with the orientation of~$K$ for even~\(k\) and is opposite otherwise.
For fixed~\(K\) and~\(m\), the isotopy class of~\(K_{\varphi}^{(m)}\) depends only on the isotopy class of \(\varphi\), hence it is completely determined by \(\fr(\varphi)\).


Let \(K\) be an \(r\)-periodic knot.
A framing \(\varphi\) of \(K\) is \(\Z_{r}\)-\emph{equivariant} if the image of \(\varphi\) is \(\Z_{r}\)-invariant in \(S^{3}\) and for any \(g \in \Z_{r}\) and we have
\[g \cdot \varphi(z,x) = \varphi\left(e^{\frac{2\pi i \ell}{r}} \cdot z, x\right), \quad \ell = 1,2,\ldots,r-1,\]
where \((z,x) \in S^{1} \times D^{2} \subset \C \times \C\).
Observe that the \(\phi\)-framed cable link of \(K\), where \(\phi\) is \(\Z_r\)-equivariant, is a \(\Z_r\)-periodic link such that every component is preserved by the action of \(\Z_r\). 
The converse is also true -- by~\cite[Theorem 8.6]{Meeks-Scott} if \(K^{(m)}_{\varphi}\) is \(r\)-periodic and each component is preserved by the group action, then \(K\) is \(r\)-periodic and \(\varphi\) is \(\Z_{r}\)-equivariant.

The next lemma characterizes equivariant framings of periodic knots in terms of the framing coefficient.

\begin{lemma}\label{lem:admissible-framings}
  Let \(K\) be an \(r\)-periodic knot.
  A framing \(\varphi\) is isotopic to a \(\Z_{r}\)-equivariant framing \(\varphi'\) if, and only if, \(\fr(\varphi)\) is divisible by \(r\).
\end{lemma}
\begin{proof}
  Let \(D\) be an \(r\)-periodic diagram of \(K\).
  The blackboard framing \(\varphi_{D}\) is \(\Z_{r}\)-equivariant and satisifies the desired condition.
  Notice that we can perform an equivariant Reidemeister I move on \(D\), i.e., we pick an orbit of points on \(D\) and perform a Reidemeister I move at each point in such a way that the result is a different \(r\)-periodic diagram of \(K\).
  Notice that such an operation changes the writhe by \(\pm r\).
  Hence, the lemma follows.
\end{proof}

Let \(p\) be a fixed prime and let \(n>0\).
Throughout this section we restrict our attention to \(p^n\)-periodic knots.
Recall that \(\mathcal{R} \subset \Z[a^{\pm1},z^{\pm1}]\) is the subring generated by \(a^{\pm1}\), \(\frac{a+a^{-1}}{z}\), and \(z\).
Moreover, \(\mathcal{J}_{p^n}\) denote the ideal in \(\mathcal{R}\) generated by monomials \(p^{i} z^{p^{n-i}}\), for \(i=0,1,\ldots,n\).
The proof of Theorem~\ref{thm:main-thm} relies on the following technical lemma.

\begin{lemma}[{\cite[Lemma 2.1., Lemma 3.2.]{Prz}}]
  \label{lemma:MnLm}
  Let $L$ be a $p^n$-periodic link.
  Suppose that $L'$ is a link obtained from $L$ by an equivariant crossing change (i.e. we perform crossing cross changes along some orbit of crossings).
  Then the following congruences are satisfied in \(\mathcal{R}\)
  \begin{align*}
    P_{L}(a,z) &\equiv P_{L'}(a,z)\ \pmod{\mathcal{J}_{p^n}}, \\
    F_{L}(a,z) &\equiv F_{L'}(a,z)\ \pmod{\mathcal{J}_{p^n}}.
  \end{align*}
\end{lemma}

\begin{proof}[Proof of Theorem~\ref{thm:main-thm}]
  Suppose that \(K\) is a \(p^n\)-periodic knot endowed with an equivariant framing \(\varphi\).
  Fix~\(m>1\) and consider the cable link \(L = K^{(m)}_{\varphi}\).
  Consider the component \(L_{m-1} = \varphi(S^{1} \times \{\frac{m-1}{m}\})\) of \(L\).
  We can perform equivariant crossing changes to unlink \(L_{m-1}\) from the other components of \(L\).
  Denote the resulting link by \(L'\).
  Observe that \(L'\) is equivalent to a split union of \(K^{(m-1)}_{\phi}\) and \(K\), hence by Lemma~\ref{lemma:MnLm}
  \begin{align*}
    P_{K^{(m)}_{\varphi}}(a,z) &\equiv \frac{a + a^{-1}}{z} \cdot P_{K^{(m-1)}_{\varphi}}(a,z) \cdot P_{K}(a,z) \pmod{\mathcal{J}_{p^{n}}}, \\
    F_{K^{(m)}_{\varphi}}(a,z) &\equiv d \cdot F_{K^{(m-1)}_{\varphi}}(a,z) \cdot F_{K}(a,z) \pmod{\mathcal{J}_{p^{n}}},
  \end{align*}
  where \(d = z^{-1}(a+a^{-1})-1\).
  Repeating the above procedure inductively, we reach the desired congruences.
\end{proof}

\section{Comparison with other criteria}
\label{sec:comp-with-other}
As pointed out in the Introduction, Theorem~\ref{thm:main-thm}, Corollary~\ref{cor:independence-of-framings}, or Corollary~\ref{cor:JoThm} can be coupled with Lemma~\ref{lem:admissible-framings} to obtain an easily applicable periodicity criterion.
In this section, we will focus on the applications of the criterion coming from Corollary~\ref{cor:JoThm}.
In particular, we will compare its strength with several previously known criteria:
\begin{enumerate}[label=(C-\arabic*)]
\item \label{item:Murasugi-criterion} the Alexander polynomial criterion of Murasugi~\cite{Mu},
\item \label{item:Naik-criterion} the Naik homological criterion~\cite{Na},
\item \label{item:Przytycki-Traczyk-criterion} the mirror image criteria of Przytycki and Traczyk~\cite{Prz,Tr},
\item \label{item:Kh-criterion} the Khovanov homology criterion of Borodzik and Politarczyk~\cite{BP}.
\end{enumerate}
The applications of the criteria listed above to the question of \(5\)-periodicity were discussed in detail in~\cite{BP}. 
In this paper, we are primarily concerned with \(7\)-periodicity.
The examples given here were chosen from the list in \cite[Example 5.3.]{BP}
The knot data was taken from SnapPy~\cite{SnapPy} and~Knot~Atlas~\cite{knotatlas}.

\subsection{\(12n749\) is not \(7\)-periodic}

\begin{figure}
  \centering
  \def\svgscale{0.7}
  \begingroup%
  \makeatletter%
  \providecommand\color[2][]{%
    \errmessage{(Inkscape) Color is used for the text in Inkscape, but the package 'color.sty' is not loaded}%
    \renewcommand\color[2][]{}%
  }%
  \providecommand\transparent[1]{%
    \errmessage{(Inkscape) Transparency is used (non-zero) for the text in Inkscape, but the package 'transparent.sty' is not loaded}%
    \renewcommand\transparent[1]{}%
  }%
  \providecommand\rotatebox[2]{#2}%
  \newcommand*\fsize{\dimexpr\f@size pt\relax}%
  \newcommand*\lineheight[1]{\fontsize{\fsize}{#1\fsize}\selectfont}%
  \ifx\svgwidth\undefined%
    \setlength{\unitlength}{242.74285215bp}%
    \ifx\svgscale\undefined%
      \relax%
    \else%
      \setlength{\unitlength}{\unitlength * \real{\svgscale}}%
    \fi%
  \else%
    \setlength{\unitlength}{\svgwidth}%
  \fi%
  \global\let\svgwidth\undefined%
  \global\let\svgscale\undefined%
  \makeatother%
  \begin{picture}(1,1.15677966)%
    \lineheight{1}%
    \setlength\tabcolsep{0pt}%
    \put(0,0){\includegraphics[width=\unitlength,page=1]{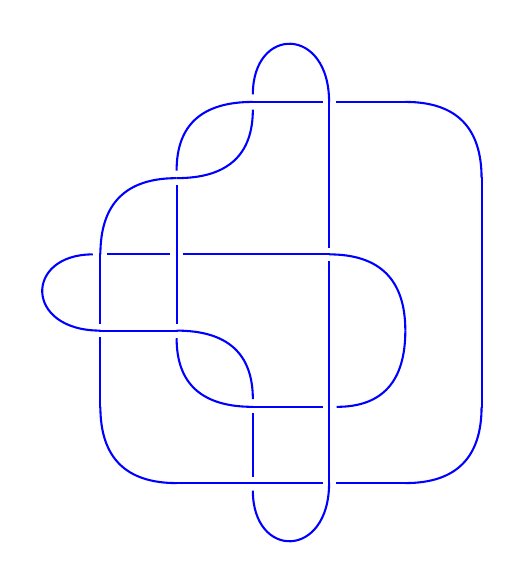}}%
  \end{picture}%
  \endgroup%
  \caption{The knot \(12n749\).}
  \label{fig:12n749}
\end{figure}

Let \(K\) be the knot \(12n749\) depicted in Figure~\ref{fig:12n749}.
We will show that \(K\) is not \(7\)-periodic.
We have
\begin{align*}
  \Delta_{K}(t) = t^{6} - t^{5} + t^{4} - t^{3} + t^{2} - t + 1 \equiv (t+1)^{6} \pmod{7},
\end{align*}
hence \(K\) satisfies~\ref{item:Murasugi-criterion}.
In particular, if \(K\) was \(7\)-periodic, the Alexander polynomial of the quotient knot would be \(\Delta_{0}=1\).
Moreover, \(H_{1}(\Sigma^{2}(K)) \cong \Z_{7}\), which implies that \(K\) satisfies~\ref{item:Naik-criterion}.
Furthermore, we have
\begin{align*}
  J_{K}(q) &= -q^{20} + q^{18} - q^{16} + q^{14} - q^{12} + q^{10} + q^{6}, \\
  P_{K}(a,z) &= a^{6} z^{6} - (a^{8} - 6 a^{6}) z^{4} + (4 a^{8} + 10 a^{6}) z^{2} - 3 a^{8} - 4 a^{6}.
\end{align*}
Consequently,
\begin{align*}
  J_{K}(q) - J_{K}(q^{-1}) &\equiv 0 \pmod{q^{14}-1}, \\
  P_{K}(a,z) - P_{K}(a^{-1},z) &\equiv 0 \pmod{(z^{7}, 7z)},
\end{align*}
where the latter congruence holds in \(\mathcal{R}\) (recall that \(\mathcal{R}\) denotes the subring of \(\Z[a^{\pm1},z^{\pm1}]\) generated by \(a^{\pm1},z,\frac{a+a^{-1}}{z}\)).
Verification of the former congruence is straightforward, while the latter congruence can be easily verified with the aid of~\cite[Lemma 1.5]{Prz}.
Consequently, \(K\) satisfies~\ref{item:Kh-criterion} and~\ref{item:Przytycki-Traczyk-criterion}.

Let \(\varphi\) be the framing of \(K\) such that \(\fr(\varphi) = 7\).
Computations using the code from Appendix~\ref{sec:code} show that
\begin{align*}
  &J_{K^{(2)}_{\varphi}}(q) - (q+q^{-1}) J_{K}^{2}(q) \equiv \\
  &\equiv 5 + 3q^{-2} + q^{-4} + 2q^{-6} + 2q^{-8} - q^{-10} - 4q^{-12} \\
  &\neq 0 \pmod{(7,q^{14}-1)}, 
\end{align*}
hence by Corollary~\ref{cor:JoThm}, \(K\) is not \(7\)-periodic.

\subsection{\(15n124640\) is not \(7\)-periodic}

\begin{figure}
  \centering
  \def\svgscale{0.7}
  \begingroup%
  \makeatletter%
  \providecommand\color[2][]{%
    \errmessage{(Inkscape) Color is used for the text in Inkscape, but the package 'color.sty' is not loaded}%
    \renewcommand\color[2][]{}%
  }%
  \providecommand\transparent[1]{%
    \errmessage{(Inkscape) Transparency is used (non-zero) for the text in Inkscape, but the package 'transparent.sty' is not loaded}%
    \renewcommand\transparent[1]{}%
  }%
  \providecommand\rotatebox[2]{#2}%
  \newcommand*\fsize{\dimexpr\f@size pt\relax}%
  \newcommand*\lineheight[1]{\fontsize{\fsize}{#1\fsize}\selectfont}%
  \ifx\svgwidth\undefined%
    \setlength{\unitlength}{325.19999964bp}%
    \ifx\svgscale\undefined%
      \relax%
    \else%
      \setlength{\unitlength}{\unitlength * \real{\svgscale}}%
    \fi%
  \else%
    \setlength{\unitlength}{\svgwidth}%
  \fi%
  \global\let\svgwidth\undefined%
  \global\let\svgscale\undefined%
  \makeatother%
  \begin{picture}(1,0.86346863)%
    \lineheight{1}%
    \setlength\tabcolsep{0pt}%
    \put(0,0){\includegraphics[width=\unitlength,page=1]{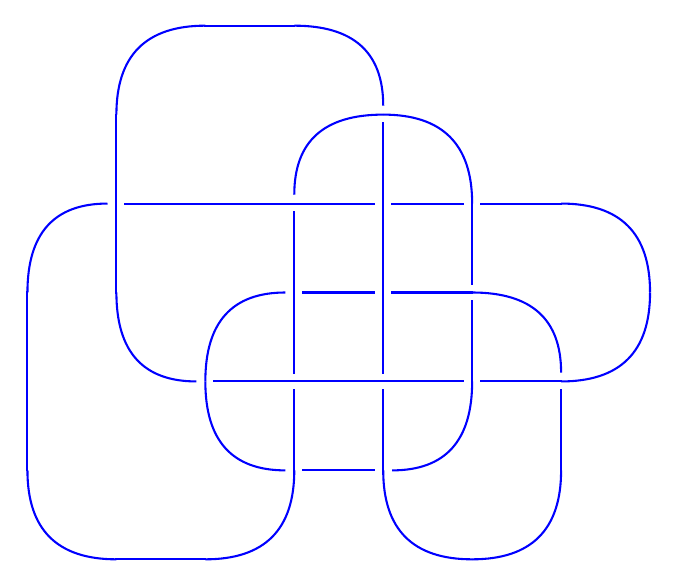}}%
  \end{picture}%
  \endgroup%
  \caption{The knot \(15n124640\).}
  \label{fig:15n124640}
\end{figure}

Consider the knot \(K = 15n124640\), see Figure~\ref{fig:15n124640}.
We will prove that \(K\) is not \(7\)-periodic.

First, notice that \(\Delta_{K}(t) = 1\), hence \(K\) satisfies~\ref{item:Murasugi-criterion} and~\ref{item:Naik-criterion}.
Furthermore, we have
\begin{align*}
  P_{K}(a,z) &= - (a^{2} + 1) z^{8} + (a^{4} + 5 a^{2} + 6 + 2 a^{-2}) z^{6} \\
             &- (2 a^{4} + 8 a^{2} + 11 + a^{-4} + 6 a^{-2}) z^{4} \\
             &+ (4 a^{-2} + a^{4} + 4 a^{2} + 6 + a^{-4}) z^{2} + 1,
\end{align*}
and an application of~\cite[Lemma 1.5]{Prz} shows that
\[P_{K}(a,z) - P_{K}(a^{-1},z) \equiv 0 \pmod{(z^{7}, 7z)},\]
hence \(K\) satisfies~\ref{item:Przytycki-Traczyk-criterion}.

Consider the \emph{Khovanov polynomial} of \(K\)
\[\operatorname{KhP}_{K}(t,q) = \sum_{i,j} t^{i} q^{j} \dim_{\Q} \operatorname{Kh}^{i,j}(K;\Q),\]
where \(\operatorname{Kh}(K;\Q)\) is the Khovanov homology of \(K\).
Using the knotkit software~\cite{KnotKit-WP, KnotKit} we obtain
\begin{equation}
  KhP_{K}(t,q) = q^{-1} + q + (1 + t q^{4}) (\mathcal{P}_{1}(t,q) + 6 \mathcal{P}_{2}(t,q)), \label{eq:decomposition-kh-criterion}
\end{equation}
for
\begin{align*}
  \mathcal{P}_{1}(t,q) &= t^{-7} q^{-13} + 4 t^{-6} q^{-11} + t^{-5} q^{-9} + 3 t^{-4} q^{-7} + t^{-3} q^{-7} + 5 t^{-3} q^{-5} \\
             &+ 4 t^{-2} q^{-5} + 5 t^{-2} q^{-3} + t^{-1} q^{-3} + 3 t^{-1} q^{-1} + 5 t q + 4 t q^3 + 5 t^2 q^3 \\
             &+ t^2 q^5 + 3 t^3q^5 + t^4q^7 + 4 t^5q^9 + t^6q^{11} + 3 q^{-1} + q,  \\
  \mathcal{P}_{2}(t,q) &= t^{-5} q^{-9} + t^{-4} q^{-7} + t^{-3} q^{-5} + t^{-2} q^{-3} + t^{-1} q^{-3} \\
             &+ t^{-1} q^{-1} + t q + t^2q^3 + t^3q^5 + t^4q^7 + q^{-1} + q.
\end{align*}
Using decomposition~\eqref{eq:decomposition-kh-criterion} we can verify that \(K\) satisfies~\ref{item:Kh-criterion}.

On the other hand, if \(\varphi\) is the framing with \(\fr(\varphi) = 0\), we obtain
\begin{align*}
  J_{K^{(2)}_{\varphi}}(q) - (q+q^{-1}) J_{K}(q)^{2} &\equiv 2 - 2q^{-2} + 2q^{-4} - 2q^{-6} + 2q^{-8} - 2q^{-10} + 2q^{-12} \\
                                                 &\neq 0 \pmod{(7, q^{14}-1)},
\end{align*}
hence by Corollary~\ref{cor:JoThm}, \(K\) is not \(7\)-periodic.

\section{Limitations}
\label{sec:limitations}

Recall that periodicity criteria~\ref{item:Przytycki-Traczyk-criterion} and~\ref{item:Kh-criterion} do not provide any information for periods \(2\), \(3\) and \(4\).
Unfortunately, the periodicity criterion from Corollary~\ref{cor:JoThm} has a similar limitation, i.e., it does not work for periods \(2\) and \(3\).
The goal of this section is to give a proof of this fact.

Let us start with the following lemma.
The lemma is well-known to the experts, but we include the proof in order to avoid confusion resulting from different conventions existing in the literature.
\begin{lemma}[Corollary~13, Theorem~14 and Theorem~15 from~\cite{jonesPolynomialInvariantKnots1985}]\label{lemma:special-values}
  Let \(L\) be a link with \(c(L)\) components.
  \begin{enumerate} 
  \item \label{item:special-values-1} Suppose that \(q \in \C\) satisfies \(q^{6}=1\).
    Then \(J_{L}(q)=J_{U_{c(L)}}(q)\), where \(U_{c(L)}\) is the \(c(L)\)-component unlink.
  \item \label{item:special-values-2} If \(q = i\), where \(i^{2}=-1\), then
    \(J_{L}(q) = \Delta_{L}(-1),\)
    where \(\Delta_{L}\) denotes the symmetrized Alexander polynomial.
  \end{enumerate}
\end{lemma}
\begin{proof}
  The lemma is a consequence of the skein relation~\eqref{eq:Jo-skein-relation}.
  In order to prove point~\ref{item:special-values-1} we will consider each case separately.
  \begin{itemize}
  \item For \(q=\pm 1\), the skein relation~\eqref{eq:Jo-skein-relation} becomes \(J_{L_{+}}(q)-J_{L_{-}}(q)= 0\), therefore the value of the Jones polynomial at \(q = \pm 1\) is invariant under crossing changes and the lemma follows.
  \item For \(q = e^{\frac{2\pi i}{3}}\) we get \(q J_{L_{+}}(q)-\bar{q} J_{L_{-}}(q)=(\bar{q}-q) J_{L_{0}}(q)\) and \(J_{U_{c(L)}}=(-1)^{c(L)-1}\), where \(\bar{q}\) denotes the complex conjugation of \(q\).
    Notice taking an orientation preserving resolution of a crossing changes the number of components by \(\pm 1\). Hence the lemma follows in this case by induction.
    Furthermore, since~\(J_{L}(q)\) is a real, it follows that \(J_{L}(\bar{q}) = J_{L}(q) = (-1)^{c(L)-1}\).
  \item For \(q = e^{\frac{\pi i}{3}}\), \(J_{U_{c(L)}}(q) = (q+q^{-1})^{c(L)-1}=1^{c(L)-1}=1\).
    The skein relation~\eqref{eq:Jo-skein-relation} reads \(\bar{q}^{2} J_{+}(q)- q^{2} J_{-}(q)=(\bar{q} - q) J_{0}(q)\).
    If follows easily that \(J_{L}(q) = 1\) satisfies the skein relation (remember that \(q+\bar{q}=1\)).
    Once more, since \(J_{L}(q)\) is real, we have \(J_{L}(\bar{q}) = J_{L}(q) = 1\).
  \end{itemize}

  For the proof of point~\ref{item:special-values-2} recall that \(\Delta_{L}(t) \in \Z[t^{\pm1/2}]\) is characterized by skein relation
  \[\Delta_{L_{+}}(t) - \Delta_{L_{-}}(t) = (t^{-1/2}-t^{1/2}) \Delta_{L_{0}}(t)\]
  together with the condition \(\Delta_{U}(t) = 1\).
  Comparison of the skein relations for \(J_{L}(i)\) and \(\Delta_{L}(-1)\) yields the desired result.
\end{proof}

Before stating the next result, recall from Corollary~\ref{cor:JoThm} that for a prime \(p\), we denote by \(\mathcal{I}_{p}\) the ideal in \(\Z[q^{\pm1}]\) generated by \(p(q^2-1)\) and \((q^2-1)^p\).

\begin{theorem}
  Let \(K\) be a knot and suppose that \(\varphi\) is a framing of \(K\) such that \(\fr(\varphi)\) is even.
  For any \(m>1\) we have
  \[J_{K^{(m)}_{\varphi}}(q) \equiv (q+q^{-1})^{m-1} J_{K}(q)^{m} \pmod{\mathcal{I}_{2}}.\]
\end{theorem}
\begin{proof}
  Let \(R = \Z[q^{\pm1}] / (q^{4}-1)\).
  Observe that
  \[\mathcal{J}_2 = \left(2 (q^2-1), (q^2-1)^2 \right) = \left(2 (q^2-1), (q^{2}-1)^{2}+2(q^{2}-1)\right) = \left( 2(q^2-1), q^4-1 \right).\]
  Consider the following epimorphisms
  \[p_{1,\pm} \colon R \to \Z, \quad p_{2} \colon R \to \Z[i],\]
  where \(p_{1,\pm}(q) = \pm1\) and \(p_{2}(q) = i\), where \(i^{2}=-1\).
  Observe that the map \(r = p_{1,+} \oplus p_{1,-} \oplus p_{2}\) is an injection.
  Indeed, injectivity of \(r\) follows from the fact that \(r \otimes \Q\) is injective.
  Therefore, it is sufficient to check that
  \begin{equation}\label{eq:first-congruence}
    J_{K^{(m)}_{\varphi}}(q) = (q+q^{-1})^{m-1} J_{K}(q)^{m}, \quad \text{for } q = \pm1,
  \end{equation}
  and the following congruence holds in \(\Z[i]\)
  \begin{equation}\label{eq:second-congruence}
    J_{K^{(m)}_{\varphi}}(i) \equiv (i+i^{-1})^{m-1} J_{K}(i)^{m} = 0 \pmod{4}.
  \end{equation}
  
  Equality~\ref{eq:first-congruence} is satisfied by Lemma~\ref{lemma:special-values}.
  To deal with~\ref{eq:second-congruence}, recall from Lemma~\ref{lemma:special-values}, that for any link \(L\) we have
  \[J_{L}(i) = \Delta_{L}(-1),\]
  where \(\Delta_{L}(t)\) denotes the symmetrized Alexander polynomial.
  Furthermore, by~\cite{Ho}, \(\Delta_{L}(t)\) is divisible by \((t-1)^{c(L)-1}\).
  Consequently, for \(m>2\), \(J_{K^{(m)}_{\varphi}}(i) = \Delta_{K^{(m)}_{\varphi}}(-1)\) is divisible by \(4\).
  For \(m=2\) it is easy to check that \(K^{(2)}_{\varphi}\) bounds a Seifert surface \(S\) of genus~\(0\) (recall that the components of \(K^{(2)}_{\varphi}\) are oppositely oriented).
  A simple calculation with the aid of the Seifert matrix of \(S\) yields \(\Delta_{K^{(2)}_{\varphi}}(t) = \pm \fr(\varphi)(t-1)\).
  Since \(\fr(\varphi)\) is even, \(\Delta_{K^{(2)}_{\varphi}}(-1)\) is divisible by \(4\), hence the theorem follows.
\end{proof}

\begin{theorem}\label{thm:ineffective for 3}
  Let $K$ be a knot and let \(\varphi\) be an arbitrary framing of \(K\).
  For any \(m \geq 2\), we have
  \[J_{K^{(m)}_{\varphi}}(q) \equiv (q+q^{-1})^{m-1}J_{K}(q)^{m} \pmod{\mathcal{I}_{3}}.\]
\end{theorem}
\begin{proof}
  Let \(R = \Z[q^{\pm1}] / (q^6-1)\).
  Observe that
  \[\mathcal{I}_3 = \left(2 (q^2-1), (q^{2}-1)^{3}+3(q^{2}-1)q^{2} \right) = \left(3 (q^2-1), (q^6-1)\right).\]
  As in the proof of the previous theorem, there is an embedding
  \[\Z[q^{\pm1}] / (q^{6}-1) \hookrightarrow \Z \oplus \Z \oplus \Z[\xi_{3}] \oplus \Z[\xi_{6}],\]
  mapping \(q\) to \((1,-1,\xi_3, \xi_6)\).
  where \(\xi_{s} = \exp \left(\frac{2\pi i}{s} \right)\), for \(s \in \Z\).
  Hence, it is sufficient to check the congruence
  \[J_{K^{(m)}_{\varphi}}(q) \equiv (q+q^{-1})^{m-1}J_{K}(q)^{m} \pmod{3(q^2-1)},\]
  for \(q = \pm1, \xi_3, \xi_6\).
  Lemma~\ref{lemma:special-values} implies that both sides of the congruence are equal to \(J_{U_{m}}=(q+q^{-1})^{m-1}\), hence the lemma follows.
\end{proof}

\appendix

\section{The code}
\label{sec:code}
The purpose of this section is to present the Sage~\cite{sagemath} code we used to verify perform calculations in Section~\ref{sec:comp-with-other}.
The code consists of two functions presented in Listing~\ref{lst:cable-function} and Listing~\ref{lst:cable-difference-function}.

\begin{lstlisting}[breaklines,language=Python,caption={The function \emph{Cable} constructs cable of a given braid with respect to the blackboard framing.},label=lst:cable-function,frame=single]
def cable(m,n,arr):
    B = BraidGroup(n*m)
    arrnew = []
    for i in arr:
        if i < 0:
            i = -i
            mult = -1
        else:
            mult = 1
        for j in xrange(1, m+1):
            for k in xrange(1, m+1):
                arrnew.append(mult*((i*m)-k+j))
    return B((arrnew))
\end{lstlisting}

The function presented in Listing~\ref{lst:cable-function} constructs the \(m\)-cable of a given braid \(b \in B_{n}\) on \(n\)-strands with respect to the blackboard framing.
The function takes three parameters as an input: an integer \(m\), an integer \(n\) and a list \(arr\) of integers.
The list \(arr\) encodes the braid \(b\).
Namely, if \(arr = (a_0,a_1,\ldots,a_k)\), where \(-n+1 \leq a_{i} \leq n-1\), for \(i=1,2,\ldots,k\), then
\[b = \sigma_{|a_0|}^{\sgn a_{0}} \cdot \sigma_{|a_1|}^{\sgn a_{1}} \cdot \ldots \cdot \sigma_{|a_{k}|}^{\sgn a_{k}} \in B_{n},\]
where \(\sigma_i\), for \(i=1,2,\ldots,n-1\) denotes the standard generators of \(B_{n}\).
The function replaces every occurence of $\sigma_{i}$, for \(i=1,2,\ldots,n-1\), in \(b\) by the braid
\[\prod_{j=0}^{m-1} \left( \sigma_{im+j} \cdot \sigma_{im+1+j} \cdot \ldots \cdot \sigma_{(i+1)m+j} \right) \in B_{nm}.\]

\begin{lstlisting}[breaklines,language=Python,numbers=left,label=lst:cable-difference-function,caption={The function \emph{Cable\textunderscore difference} verifies periodicity criterion from Corollary~\ref{cor:JoThm}. },frame=single]
def cable_difference(p,m,n,arr):
    R.<A> = LaurentPolynomialRing(ZZ)
    B = BraidGroup(n)
    K = Link(B(arr))
    assert K.is_knot()
    w = K.writhe() % p
    if (2*w) < p:
        for i in xrange(w):
            arr.append(-n-i)
        n = n+w
    else:
        for i in xrange(p-w):
            arr.append(n+i)
        n = n+p-w
    B = BraidGroup(n)
    J1 = K.jones_polynomial(skein_normalization=True)
    J2 = Cable(m,n,arr).jones_polynomial(skein_normalization=True)
    w = 6*(L.writhe())*(m//2)
    J2 = J2*(A^w)
    diff = ((-A^2-A^(-2))^(m-1)) * (J1)^m - J2
    diffp = diff.polynomial_construction()[0]
    diffp = diffp.quo_rem(A^(4*p)-1)[1]
    diffp = diffp.quo_rem(p)[1]
    return diffp
\end{lstlisting}

The function from Listing~\ref{lst:cable-difference-function} implements periodicity criterion from Corollary~\ref{cor:JoThm}.
It takes as an input an integer~\(p\) (the potential period), an integer~\(m\) (the multiplicity of the cable) and an integer~\(n\) together with a list of integers \(arr\) which encodes a braid on \(n\)-strands representing a knot \(K\).
The function operates not on the Jones polynomial itself, but on the oriented Kauffman bracket version thereof
\[\widetilde{J}_{K}(A) = A^{-6w} \langle K \rangle \in \Z[A^{\pm1}] = J_{K}(A^{2}),\]
hence it returns the difference
\[\widetilde{J}_{K^{(m)}}(A) - (-A^{2}-A^{-2})^{m-1} \widetilde{J}_{K}^{m}(A) \pmod{(p,A^{4p}-1)}\]
As the first step, the function fixes the framing (lines \(6\)-\(14\)) by performing an appropriate number of Reidemeister I moves (in the form of Markov moves on the braid) so that the writhe is divisible by \(p\), see Corollary~\ref{lem:admissible-framings}.
Finally, it performs the periodicity test and returns the result.

\bibliographystyle{abbrv}
\bibliography{bibliography}

\end{document}